\documentclass[11pt,twoside]{amsart}
\usepackage{graphicx}
\usepackage{color}
\usepackage{amssymb}
\usepackage{amsmath}
\usepackage{amsthm}
\usepackage{amsfonts}
\usepackage{amsmath, amsfonts, amssymb}
\newtheorem{theorem}{Theorem}[section]

\newtheorem{lemma}[theorem]{Lemma}

\newtheorem{definition}[theorem]{Definition}
\numberwithin{equation}{section}

\begin{document}
	\small \begin{center}{\textit{ In the name of
				Allah, the Beneficent, the Merciful.}}\end{center}
	\vspace{0.2cm}
	\large
	\title{On three-dimensional anti-commutative algebras}
	\author{U. Bekbaev}
		\thanks{{\scriptsize
			emails: uralbekbaev@gmail.com}}
	\maketitle
	\begin{center}
		\address{Turin Polytechnic University in Tashkent, Tashkent, Uzbekistan.}
	\end{center}
	
	\begin{abstract}
		This paper is devoted to the classification problem of tree-dimensional anti-commutative(zero-potent) algebras over any base field $\mathbb{F}$ such that $Char(\mathbb{F})\neq 2$ and every element admits a square root. 
		
	\end{abstract}
	
	\section{Introduction}
	
	Anti-commutative (zero-potent) algebras are among the algebraic structures closest to Lie algebras. In this paper, we study the classification problem of three-dimensional anti-commutative algebras over an arbitrary base field $\mathbb{F}$ such that $\operatorname{char}(\mathbb{F})\neq 2$ and every element of $\mathbb{F}$ admits a square root. 
	
	In particular, such a classification includes the classification of three-dimensional Lie algebras over $\mathbb{F}$ as a special case, which is a classical result. Therefore, in this paper we focus mainly on the non-Lie algebra case. We compare our results with the most recent classification obtained in \cite{K}. Finally, we describe all automorphisms and derivations of the canonical representatives of the non-Lie algebras.
	
	Our approach differs from the existing ones and provides a shorter and more direct solution.
	\section{The Main Result}
	
	\begin{definition}
		A vector space $\mathbb{V}$ over a field $\mathbb{F}$ equipped with a bilinear map
		\[
		\cdot : \mathbb{V}\otimes \mathbb{V} \rightarrow \mathbb{V},
		\qquad
		(\mathrm{x},\mathrm{y}) \mapsto \mathrm{x}\cdot \mathrm{y},
		\]
		satisfying
		\[
		(\alpha \mathrm{x} + \beta \mathrm{y}) \cdot \mathrm{z}
		= \alpha (\mathrm{x}\cdot \mathrm{z}) + \beta (\mathrm{y}\cdot \mathrm{z}),
		\qquad
		\mathrm{z}\cdot (\alpha \mathrm{x} + \beta \mathrm{y})
		= \alpha (\mathrm{z}\cdot \mathrm{x}) + \beta (\mathrm{z}\cdot \mathrm{y}),
		\]
		for all $\mathrm{x}, \mathrm{y}, \mathrm{z}\in \mathbb{V}$ and $\alpha, \beta \in \mathbb{F}$,
		is called an algebra $\mathbb{A}=(\mathbb{V},\cdot)$.
	\end{definition}
	
	Let $\mathbb{A}=(\mathbb{V},\cdot)$ be an $n$-dimensional algebra over $\mathbb{F}$ and
	$\mathbf{e}=(\mathrm{e}_1,\mathrm{e}_2,\ldots,\mathrm{e}_n)$ be a fixed basis.
	Then the bilinear map $\cdot : \mathbb{V}\times\mathbb{V}\to\mathbb{V}$ is represented
	by an $n\times n^2$ matrix (called the matrix of structure constants, abbreviated MSC)
	\[
	A=
	\begin{pmatrix}
		a_{11}^1 & a_{12}^1 & \cdots & a_{1n}^1 & a_{21}^1 & \cdots & a_{nn}^1\\
		a_{11}^2 & a_{12}^2 & \cdots & a_{1n}^2 & a_{21}^2 & \cdots & a_{nn}^2\\
		\vdots   & \vdots   &        & \vdots   & \vdots   &        & \vdots\\
		a_{11}^n & a_{12}^n & \cdots & a_{1n}^n & a_{21}^n & \cdots & a_{nn}^n
	\end{pmatrix},
	\]
	defined by
	\[
	\mathrm{e}_i\cdot \mathrm{e}_j
	= \sum_{k=1}^n a_{ij}^k\,\mathrm{e}_k,
	\qquad i,j=1,2,\ldots,n.
	\]
	
	Thus, once a basis $\mathbf{e}$ is fixed, the algebra $\mathbb{A}$ can be identified with
	its MSC $A\in M_{n\times n^2}(\mathbb{F})$.
	
	The product on $\mathbb{A}$ with respect to the basis $\mathbf{e}$ can be written as
	\[
	\mathrm{x}\cdot \mathrm{y} = \mathbf{e}\, A (x\otimes y),
	\]
	for any $\mathrm{x}=\mathbf{e}x$ and $\mathrm{y}=\mathbf{e}y$, where
	$x=(x_1,\ldots,x_n)^T$ and $y=(y_1,\ldots,y_n)^T$ are the coordinate column vectors
	of $\mathrm{x}$ and $\mathrm{y}$, respectively, and $x\otimes y$ denotes the
	Kronecker (tensor) product.
	
	In terms of MSC, an automorphism $\mathbf{g}$ and a derivation $\mathbf{D}$ of the algebra
	$\mathbb{A}$ are characterized by
	\[
	A = g^{-1}A(g\otimes g),
	\qquad
	DA = A(D\otimes I + I\otimes D),
	\]
	where $g$ (respectively, $D$) is the matrix of $\mathbf{g}$ (respectively, $\mathbf{D}$)
	in the chosen basis.
	
	\begin{definition}
		Two algebras with MSCs $A$ and $B$ are said to be isomorphic if there exists
		$g\in GL_n(\mathbb{F})$ such that
		\[
		B = gA(g^{-1}\otimes g^{-1}).
		\]
	\end{definition}
	
	Therefore, the classification of $n$-dimensional algebras up to isomorphism is equivalent
	to the classification of $n\times n^2$ matrices under the action
	\[
	A \mapsto gA(g^{-1}\otimes g^{-1}), \qquad g\in GL_n(\mathbb{F}).
	\]
	
	Once a basis is fixed, we do not distinguish between an algebra and its MSC, or between
	a linear map and its representing matrix.
	
	\begin{definition}
		An algebra $\mathbb{A}$ is called anti-commutative (zero-potent) if
		\[
		a^2 = 0 \quad \text{for all } a\in \mathbb{A}.
		\]
	\end{definition}
	
	In the three-dimensional case, we write the MSC in the form
	\[
	A=
	\begin{pmatrix}
		a_1 & a_2 & a_3 & a_4 & a_5 & a_6 & a_7 & a_8 & a_9\\
		b_1 & b_2 & b_3 & b_4 & b_5 & b_6 & b_7 & b_8 & b_9\\
		c_1 & c_2 & c_3 & c_4 & c_5 & c_6 & c_7 & c_8 & c_9
	\end{pmatrix}.
	\]
	
	Throughout the paper, $\mathbb{F}$ denotes a field of characteristic
	$\operatorname{char}(\mathbb{F})\neq 2$ in which every element admits a square root.
	
	\begin{lemma}
		Every three-dimensional anti-commutative algebra has a two-dimensional subalgebra.
	\end{lemma}
	
	\begin{proof}
		Assume that a three-dimensional anti-commutative algebra $\mathbb{A}$ has no
		two-dimensional subalgebra. Then, for any linearly independent $e_1,e_2\in\mathbb{A}$,
		the elements $e_1,e_2,e_3$ are linearly independent, where $e_3=e_1e_2$.
		Write
		\[
		e_2e_3 = a e_1 + b e_2 + c e_3.
		\]
		
		Consider the linearly independent vectors $f_1=e_2$ and $f_2=xe_1+e_3$.
		Then
		\[
		f_1f_2
		= e_2(xe_1+e_3)
		= -x e_3 + a e_1 + b e_2 + c e_3
		= a e_1 + b e_2 + (c-x)e_3.
		\]
		Hence,
		\[
		\det
		\begin{pmatrix}
			0 & 1 & 0\\
			x & 0 & 1\\
			a & b & c-x
		\end{pmatrix}
		= x^2 - cx + a,
		\]
		which can be made zero by an appropriate choice of $x$. Thus,
		$f_1,f_2,f_1f_2$ are linearly dependent, yielding a contradiction.
	\end{proof}
	
	The following classification of three-dimensional Lie algebras over $\mathbb{F}$
	is well known \cite{W,P}.
	
	\begin{theorem}
		Any non-trivial three-dimensional Lie algebra over $\mathbb{F}$ is isomorphic to
		exactly one of the following algebras, listed by their matrices of structure constants:
		\[
		\mathfrak{sl}_2(\mathbb{F}) =
		\begin{pmatrix}
			0 & 0 & -2 & 0 & 0 & 0 & 2 & 0 & 0\\
			0 & 0 & 0  & 0 & 0 & 2 & 0 & -2 & 0\\
			0 & 1 & 0  & -1& 0 & 0 & 0 & 0 & 0
		\end{pmatrix},
		\]
		where
		\[
		e=\begin{pmatrix}0&1\\0&0\end{pmatrix},\quad
		f=\begin{pmatrix}0&0\\1&0\end{pmatrix},\quad
		h=\begin{pmatrix}1&0\\0&-1\end{pmatrix},
		\]
		\[
		\mathfrak{r}_{3,\lambda}=
		\begin{pmatrix}
			0 & 0 & 0 & 0 & 0 & 0 & 0 & 0 & 0\\
			0 & 1 & 0 & -1 & 0 & 0 & 0 & 0 & 0\\
			0 & 0 & \lambda & 0 & 0 & 0 & -\lambda & 0 & 0
		\end{pmatrix},
		\quad \lambda\in\mathbb{F}, \quad
		\mathfrak{r}_{3,\lambda}\simeq \mathfrak{r}_{3,1/\lambda}\ (\lambda\neq 0),
		\]
		\[
		\mathfrak{r}'_{3,1}=
		\begin{pmatrix}
			0 & 0 & 0 & 0 & 0 & 0 & 0 & 0 & 0\\
			0 & 1 & 1 & -1 & 0 & 0 & -1 & 0 & 0\\
			0 & 0 & 1 & 0 & 0 & 0 & -1 & 0 & 0
		\end{pmatrix},
		\]
		\[
		\mathfrak{h}_3=
		\begin{pmatrix}
			0 & 0 & 0 & 0 & 0 & 0 & 0 & 0 & 0\\
			0 & 0 & 0 & 0 & 0 & 0 & 0 & 0 & 0\\
			0 & 0 & 0 & 0 & 0 & 1 & 0 & -1 & 0
		\end{pmatrix},
		\]
		the Heisenberg algebra.
	\end{theorem}
		The following theorem constitutes the main result of this paper. 
	Unless explicitly stated otherwise, all listed algebras are non-Lie.
	
	\begin{theorem}
		Any non-trivial three-dimensional anti-commutative algebra over $\mathbb{F}$
		is isomorphic to exactly one of the following algebras, listed by their matrices
		of structure constants:
		\[
		A_1 =
		\begin{pmatrix}
			0 & 0 & 0 & 0 & 0 & 0 & 0 & 0 & 0\\
			0 & 1 & 0 & -1 & 0 & 0 & 0 & 0 & 0\\
			0 & 0 & 0 & 0 & 0 & 1 & 0 & -1 & 0
		\end{pmatrix},
		\]
		\[
		A_2(\lambda)=
		\begin{pmatrix}
			0 & 0 & 1 & 0 & 0 & 0 & -1 & 0 & 0\\
			0 & 1 & \lambda & -1 & 0 & 0 & -\lambda & 0 & 0\\
			0 & 0 & 0 & 0 & 0 & 1 & 0 & -1 & 0
		\end{pmatrix},\quad \lambda\in\mathbb{F},
		\]
		\[
		A_3=
		\begin{pmatrix}
			0 & 0 & 0 & 0 & 0 & 0 & 0 & 0 & 0\\
			0 & 1 & 1 & -1 & 0 & 0 & -1 & 0 & 0\\
			0 & 0 & 0 & 0 & 0 & 1 & 0 & -1 & 0
		\end{pmatrix},
		\]
		\[
		A_4(\lambda)=
		\begin{pmatrix}
			0 & 0 & 0 & 0 & 0 & 1 & 0 & -1 & 0\\
			0 & 1 & 0 & -1 & 0 & 0 & 0 & 0 & 0\\
			0 & 0 & \lambda & 0 & 0 & 0 & -\lambda & 0 & 0
		\end{pmatrix},\quad \lambda\in\mathbb{F},
		\]
		where $A_4(-1)\simeq \mathfrak{sl}_2$,
		\[
		A_5=
		\begin{pmatrix}
			0 & 0 & 0 & 0 & 0 & 1 & 0 & -1 & 0\\
			0 & 1 & 0 & -1 & 0 & 1 & 0 & -1 & 0\\
			0 & 0 & 1 & 0 & 0 & 0 & -1 & 0 & 0
		\end{pmatrix},
		\]
		\[
		A_6(\lambda)=
		\begin{pmatrix}
			0 & 0 & 0 & 0 & 0 & 0 & 0 & 0 & 0\\
			0 & 1 & 0 & -1 & 0 & 0 & 0 & 0 & 0\\
			0 & 0 & \lambda & 0 & 0 & 0 & -\lambda & 0 & 0
		\end{pmatrix}
		=\mathfrak{r}_{3,\lambda},
		\]
		with $A_6(\lambda)\simeq A_6(1/\lambda)$ for $\lambda\neq 0$,
		\[
		A_7(\lambda)=
		\begin{pmatrix}
			0 & 0 & 1 & 0 & 0 & 0 & -1 & 0 & 0\\
			0 & 1 & 0 & -1 & 0 & 0 & 0 & 0 & 0\\
			0 & 0 & \lambda & 0 & 0 & 0 & -\lambda & 0 & 0
		\end{pmatrix},\quad \lambda\in\mathbb{F},
		\]
		\[
		A_8=
		\begin{pmatrix}
			0 & 0 & 0 & 0 & 0 & 0 & 0 & 0 & 0\\
			0 & 1 & 1 & -1 & 0 & 0 & -1 & 0 & 0\\
			0 & 0 & 1 & 0 & 0 & 0 & -1 & 0 & 0
		\end{pmatrix}
		=\mathfrak{r}'_{3,1},
		\]
		\[
		A_9=
		\begin{pmatrix}
			0 & 0 & 0 & 0 & 0 & 1 & 0 & -1 & 0\\
			0 & 0 & 0 & 0 & 0 & 0 & 0 & 0 & 0\\
			0 & 0 & 0 & 0 & 0 & 0 & 0 & 0 & 0
		\end{pmatrix}
		=\mathfrak{h}_3,
		\]
		the Heisenberg algebra.
	\end{theorem}
	\begin{proof}
		We first classify three-dimensional anti-commutative algebras that admit a
		non-trivial two-dimensional subalgebra.
		For this purpose, we use the following well-known fact concerning
		two-dimensional algebras.
		
		Every non-trivial two-dimensional anti-commutative algebra is isomorphic to
		the algebra given by the matrix of structure constants
		\[
		A=
		\begin{pmatrix}
			0 & 0 & 0 & 0\\
			0 & 1 & -1 & 0
		\end{pmatrix},
		\]
		and its automorphism group is
		\[
		\mathrm{Aut}(A)=
		\left\{
		\begin{pmatrix}
			1 & 0\\
			c & d
		\end{pmatrix}
		:\ c,d\in\mathbb{F},\ d\neq 0
		\right\}.
		\]
		
		To classify all such three-dimensional algebras, it is sufficient to consider
		algebras with matrices of structure constants of the form
		\[
		A=
		\begin{pmatrix}
			0 & 0 & a_3 & 0 & 0 & a_6 & -a_3 & -a_6 & 0\\
			0 & 1 & b_3 & -1 & 0 & b_6 & -b_3 & -b_6 & 0\\
			0 & 0 & c_3 & 0 & 0 & c_6 & -c_3 & -c_6 & 0
		\end{pmatrix},
		\]
		and to study their isomorphisms under the action of the group
		\[
		G_1=
		\left\{
		g=
		\begin{pmatrix}
			1 & 0 & x\\
			c & d & y\\
			0 & 0 & z
		\end{pmatrix}
		:\ c,d,x,y,z\in\mathbb{F},\ dz\neq 0
		\right\},
		\]
		which consists of transformations preserving the subalgebra generated by
		$e_1,e_2$.
		
		For $A' = g^{-1}A(g\otimes g)$, we obtain
		\[
		A'=
		\begin{pmatrix}
			0 & 0 & M_{13} & 0 & 0 & M_{23} & -M_{13} & -M_{23} & 0\\
			0 & 1 & N_{13} & -1 & 0 & N_{23} & -N_{13} & -N_{23} & 0\\
			0 & 0 & P_{13} & 0 & 0 & P_{23} & -P_{13} & -P_{23} & 0
		\end{pmatrix},
		\]
		where
		\begin{align*}
			M_{13} &= z(a_3 + c a_6) - x(c_3 + c c_6), &
			M_{23} &= d(z a_6 - x c_6),\\
			N_{13} &= \frac{1}{d}\!\left[(1-(c_3+c c_6))(y-cx)
			+ z(b_3 + c b_6 - c a_3 - c^2 a_6)\right], &
			N_{23} &= z(b_6 - c a_6) - x(1 - c c_6) - c_6 y,\\
			P_{13} &= c_3 + c c_6, &
			P_{23} &= d c_6.
		\end{align*}
		
		Assume first that $c_6\neq 0$.
		Choosing
		\[
		d=\frac{1}{c_6},\quad
		c=-\frac{c_3}{c_6},\quad
		x=\frac{z a_6}{c_6},\quad
		y=\frac{z(b_6 - c a_6) - x(1 - c c_6)}{c_6},
		\]
		we obtain $P_{23}=1$, $P_{13}=0$, $M_{23}=0$, and $N_{23}=0$.
		Thus, we may assume that
		\[
		A=
		\begin{pmatrix}
			0 & 0 & a_3 & 0 & 0 & 0 & -a_3 & 0 & 0\\
			0 & 1 & b_3 & -1 & 0 & 0 & -b_3 & 0 & 0\\
			0 & 0 & 0 & 0 & 0 & 1 & 0 & -1 & 0
		\end{pmatrix}.
		\]
		
		If $a_3=b_3=0$, then choosing $x=y=0$, $c=-c_3$, and $d=1$ transforms $A$ into
		\[
		A_1=
		\begin{pmatrix}
			0 & 0 & 0 & 0 & 0 & 0 & 0 & 0 & 0\\
			0 & 1 & 0 & -1 & 0 & 0 & 0 & 0 & 0\\
			0 & 0 & 0 & 0 & 0 & 1 & 0 & -1 & 0
		\end{pmatrix},
		\]
		which is a non-Lie algebra.
		
		If $a_3\neq 0$, then choosing $x=y=0$, $z=1/a_3$, $c=-c_3$, and $d=1$
		transforms $A$ into
		\[
		A_2=
		\begin{pmatrix}
			0 & 0 & 1 & 0 & 0 & 0 & -1 & 0 & 0\\
			0 & 1 & a & -1 & 0 & 0 & -a & 0 & 0\\
			0 & 0 & 0 & 0 & 0 & 1 & 0 & -1 & 0
		\end{pmatrix},
		\]
		which is again non-Lie.
		
		If $a_3=0$ and $b_3\neq 0$, then $A$ can be transformed into
		\[
		A_3=
		\begin{pmatrix}
			0 & 0 & 0 & 0 & 0 & 0 & 0 & 0 & 0\\
			0 & 1 & 1 & -1 & 0 & 0 & -1 & 0 & 0\\
			0 & 0 & 0 & 0 & 0 & 1 & 0 & -1 & 0
		\end{pmatrix},
		\]
		which is also non-Lie.
		
		Next, assume $c_6=0$. In this case,
		\begin{align*}
			M_{13} &= z(a_3 + c a_6) - x c_3, &
			M_{23} &= d z a_6,\\
			N_{13} &= \frac{1}{d}\!\left[(1-c_3)(y-cx)
			+ z(b_3 + c b_6 - c a_3 - c^2 a_6)\right], &
			N_{23} &= z(b_6 - c a_6) - x,\\
			P_{13} &= c_3, &
			P_{23} &= 0.
		\end{align*}
		
		If $a_6\neq 0$ and $c_3\neq 1$, then $A$ can be transformed into
		\[
		A_4=
		\begin{pmatrix}
			0 & 0 & 0 & 0 & 0 & 1 & 0 & -1 & 0\\
			0 & 1 & 0 & -1 & 0 & 0 & 0 & 0 & 0\\
			0 & 0 & a & 0 & 0 & 0 & -a & 0 & 0
		\end{pmatrix},\quad a\neq 1.
		\]
		This algebra is Lie if and only if $a=-1$.
		The change of basis $e=e_2$, $f=-2e_3$, $h=2e_1$ shows that
		$A_4(-1)\simeq \mathfrak{sl}_2$.
		
		If $a_6\neq 0$ and $c_3=1$, then the system reduces to
		\begin{align*}
			M_{13} &= z(a_3 + c a_6) - x, &
			M_{23} &= d z a_6,\\
			N_{13} &= \frac{z}{d}(b_3 + c b_6 - c a_3 - c^2 a_6), &
			N_{23} &= z(b_6 - c a_6) - x,\\
			P_{13} &= 1, &
			P_{23} &= 0,
		\end{align*}
		and $A$ can be transformed into
		\[
		A_5(a)=
		\begin{pmatrix}
			0 & 0 & 0 & 0 & 0 & 1 & 0 & -1 & 0\\
			0 & 1 & 0 & -1 & 0 & a & 0 & -a & 0\\
			0 & 0 & 1 & 0 & 0 & 0 & -1 & 0 & 0
		\end{pmatrix},
		\]
		where $A_5(0)=A_4(1)$.
		
		The algebra $A_5(a)$ is not Lie, since
		\[
		(e_1e_2)e_3 + (e_2e_3)e_1 + (e_3e_1)e_2
		= 2e_1 + a e_2 \neq 0.
		\]
		Moreover, the equality
		\[
		A_5(a)=g^{-1}A_5(b)(g\otimes g),
		\qquad
		g=\mathrm{Diag}\!\left(1,\frac{b}{a},\frac{a}{b}\right),
		\]
		shows that, up to isomorphism, there is only one such algebra,
		namely
		\[
		A_5=
		\begin{pmatrix}
			0 & 0 & 0 & 0 & 0 & 1 & 0 & -1 & 0\\
			0 & 1 & 0 & -1 & 0 & 1 & 0 & -1 & 0\\
			0 & 0 & 1 & 0 & 0 & 0 & -1 & 0 & 0.
		\end{pmatrix}
		\]
		If $a_6=0$, then the system becomes
		\begin{align*}
			M_{13} &= z a_3 - x c_3, & M_{23} &= 0,\\
			N_{13} &= \frac{1}{d}\left[(1-c_3)(y-cx)+z(b_3+c b_6-c a_3)\right], &
			N_{23} &= z b_6 - x,\\
			P_{13} &= c_3, & P_{23} &= 0.
		\end{align*}
		Hence, choosing $x=z b_6$ reduces the system to
		\begin{align*}
			M_{13} &= z(a_3-b_6 c_3), & M_{23} &= 0,\\
			N_{13} &= \frac{1}{d}\left[(1-c_3)(y-z b_6 c)+z(b_3+c b_6-c a_3)\right], &
			N_{23} &= 0,\\
			P_{13} &= c_3, & P_{23} &= 0.
		\end{align*}
		
		If $c_3\neq 1$, one can make $N_{13}=0$ and normalize $M_{13}$ to either $0$ or $1$.
		Thus, we obtain
		\[
		A_6(a)=
		\begin{pmatrix}
			0 & 0 & 0 & 0 & 0 & 0 & 0 & 0 & 0\\
			0 & 1 & 0 & -1 & 0 & 0 & 0 & 0 & 0\\
			0 & 0 & a & 0 & 0 & 0 & -a & 0 & 0
		\end{pmatrix}
		=\mathfrak{r}_{3,a},\quad a\neq 1,
		\]
		and
		\[
		A_7(a)=
		\begin{pmatrix}
			0 & 0 & 1 & 0 & 0 & 0 & -1 & 0 & 0\\
			0 & 1 & 0 & -1 & 0 & 0 & 0 & 0 & 0\\
			0 & 0 & a & 0 & 0 & 0 & -a & 0 & 0
		\end{pmatrix},\quad a\neq 1.
		\]
		The algebra $A_7(a)$ is not Lie, since
		\[
		(e_1e_2)e_3+(e_2e_3)e_1+(e_3e_1)e_2
		=e_2e_3-(e_1+a e_3)e_2=-e_2\neq 0.
		\]
		
		If $c_3=1$, then one can normalize $M_{13}=1$ and $N_{13}=0$ in the case $b_6\neq a_3$;
		otherwise, one can achieve $M_{13}=0$ and $N_{13}=0$ or $N_{13}=1$.
		This yields the following additional algebras:
		\[
		A=
		\begin{pmatrix}
			0 & 0 & 1 & 0 & 0 & 0 & -1 & 0 & 0\\
			0 & 1 & 0 & -1 & 0 & 0 & 0 & 0 & 0\\
			0 & 0 & 1 & 0 & 0 & 0 & -1 & 0 & 0
		\end{pmatrix}
		=A_7(1),
		\]
		\[
		A=
		\begin{pmatrix}
			0 & 0 & 0 & 0 & 0 & 0 & 0 & 0 & 0\\
			0 & 1 & 0 & -1 & 0 & 0 & 0 & 0 & 0\\
			0 & 0 & 1 & 0 & 0 & 0 & -1 & 0 & 0
		\end{pmatrix}
		=A_6(1)=\mathfrak{r}_{3,1},
		\]
		and
		\[
		A_8=
		\begin{pmatrix}
			0 & 0 & 0 & 0 & 0 & 0 & 0 & 0 & 0\\
			0 & 1 & 1 & -1 & 0 & 0 & -1 & 0 & 0\\
			0 & 0 & 1 & 0 & 0 & 0 & -1 & 0 & 0
		\end{pmatrix}
		=\mathfrak{r}'_{3,1}.
		\]
		
		Next, we consider three-dimensional anti-commutative algebras whose only
		two-dimensional subalgebras are trivial.
		We compute $A'=g^{-1}A(g\otimes g)$, where
		\[
		A=
		\begin{pmatrix}
			0 & 0 & a_3 & 0 & 0 & a_6 & -a_3 & -a_6 & 0\\
			0 & 0 & b_3 & 0 & 0 & b_6 & -b_3 & -b_6 & 0\\
			0 & 0 & c_3 & 0 & 0 & c_6 & -c_3 & -c_6 & 0
		\end{pmatrix},
		\quad
		g=
		\begin{pmatrix}
			a & b & x\\
			c & d & y\\
			0 & 0 & z
		\end{pmatrix},
		\]
		a transformation preserving the subalgebra generated by $e_1,e_2$.
		
		We may assume that if $b_3=0$, then $a_3=c_3=0$, and if $a_6=0$, then
		$b_6=c_6=0$; otherwise, the algebra would contain a non-trivial
		two-dimensional subalgebra.
		
		One computes
		\[
		A'=
		\begin{pmatrix}
			0 & 0 & C_1 & 0 & 0 & D_1 & -C_1 & -D_1 & E_1\\
			0 & 0 & C_2 & 0 & 0 & D_2 & -C_2 & -D_2 & E_2\\
			0 & 0 & C_3 & 0 & 0 & D_3 & -C_3 & -D_3 & E_3
		\end{pmatrix},
		\]
		where $\Delta_0=(ad-bc)z\neq 0$ and the coefficients are given explicitly
		(as above).
		
		Substituting $x=y=0$ yields
		\[
		\begin{pmatrix}
			C_1 & D_1\\
			C_2 & D_2
		\end{pmatrix}
		=
		\begin{pmatrix}
			a & b\\
			c & d
		\end{pmatrix}^{-1}
		\begin{pmatrix}
			a_3 & a_6\\
			b_3 & b_6
		\end{pmatrix}
		\begin{pmatrix}
			a & b\\
			c & d
		\end{pmatrix},
		\qquad
		(C_3,D_3)=(c_3,c_6)
		\begin{pmatrix}
			a & b\\
			c & d
		\end{pmatrix}.
		\]
		
		There are only the following possibilities:
		\begin{enumerate}
			\item $\begin{pmatrix}C_1&D_1\\C_2&D_2\end{pmatrix}=
			\begin{pmatrix}\lambda_1&0\\0&\lambda_2\end{pmatrix}$,
			$\lambda_1\neq\lambda_2$, with $C_3,D_3\in\{0,1\}$;
			\item $\begin{pmatrix}C_1&D_1\\C_2&D_2\end{pmatrix}=
			\lambda I$, with $(C_3,D_3)=(1,0)$ or $(0,0)$;
			\item $\begin{pmatrix}C_1&D_1\\C_2&D_2\end{pmatrix}=
			\begin{pmatrix}\lambda&1\\0&\lambda\end{pmatrix}$,
			with $(C_3,D_3)=(a,1)$, $(1,0)$, or $(0,0)$.
		\end{enumerate}
		
		In order to avoid non-trivial two-dimensional subalgebras, one must have
		$(C_1,D_2)=(0,0)$.
		In this case, all resulting algebras are Lie algebras, and among them only
		\[
		A_9=
		\begin{pmatrix}
			0 & 0 & 0 & 0 & 0 & 1 & 0 & -1 & 0\\
			0 & 0 & 0 & 0 & 0 & 0 & 0 & 0 & 0\\
			0 & 0 & 0 & 0 & 0 & 0 & 0 & 0 & 0
		\end{pmatrix},
		\]
		corresponding to
		$\begin{pmatrix}C_1&D_1\\C_2&D_2\end{pmatrix}=
		\begin{pmatrix}0&1\\0&0\end{pmatrix}$,
		has no non-trivial two-dimensional subalgebra.
		This algebra is the Heisenberg algebra $\mathfrak{h}_3$.
		
		To show that the non-Lie algebras
		$A_1$, $A_2(\lambda)$, $A_3$, $A_4(\lambda)$, $A_5$, and $A_7(\lambda)$
		are pairwise non-isomorphic, one may use standard invariants such as
		$\dim Z(A)$, $\dim\mathrm{Ann}(A)$, $\dim A^2$, and
		$\dim\mathrm{InnDer}(A)$.
		
		Among these, only $A_1$ has a non-trivial center,
		$Z(A_1)=\langle e_1\rangle$, and thus it is not isomorphic to the others.
		Only $A_7(\lambda)$ has two-dimensional annihilator, and therefore it is
		also not isomorphic to the remaining algebras.
		
		The following table shows that
		$A_2(\lambda)$, $A_3$, $A_4(\lambda)$, and $A_5$
		are mutually non-isomorphic:
		\[
		\begin{array}{c|ccc}
			A & \dim\mathrm{Ann}(A) & \dim A^2 & \dim(\mathrm{Ann}(A)\cap A^2)\\ \hline
			A_2(\lambda) & 0 & 3 & 0\\
			A_3 & 1 & 2 & 1\\
			A_4(\lambda) & 1 & 2 & 0\\
			A_5 & 1 & 3 & 1
		\end{array}
		\]
	Of course, in each case $i = 2,4,7$ there arises the question whether the algebras $A_i(\lambda)$ corresponding to different values of $ \lambda$ are mutually isomorphic or not. So far, our analysis shows that the answer is \textbf{no}. In this context, the study of derivations and automorphisms of these algebras may be useful.
		\end{proof}
		In \cite{K}, one can find another classification of three-dimensional anti-commutative non-Lie algebras over $\mathbb{F}$. The MSCs of their canonical representatives, written in our notation, are as follows:
	\begin{itemize}
		\item	$ZA_1=\begin{pmatrix}0&0&0&0&0&0&0&0&0\\ 
			0&0&0&0&0&1&0&-1&0\\
			0&1&0&-1&0&0&0&0&0
		\end{pmatrix}$,
		\item $ZA_2=\begin{pmatrix}0&0&0&0&0&0&0&0&0\\ 
			0&0&0&0&0&1&0&-1&0\\
			0&1&-1&-1&0&1&1&-1&0
		\end{pmatrix}$,
		\item $ZA_3(a)=\begin{pmatrix}0&0&0&0&0&1&0&-1&0\\ 
			0&0&-1&0&0&\lambda&1&-\lambda&0\\
			0&1&0&-1&0&0&0&0&0
		\end{pmatrix}, 0\neq \lambda\in \mathbb{F}, $
		\item $ZA_4=\begin{pmatrix}0&0&-1&0&0&1&1&-1&0\\ 
			0&0&-2&0&0&2&2&-2&0\\
			0&1&0&-1&0&2&0&-2&0
		\end{pmatrix},$
		\item $ZA_5=\begin{pmatrix}0&0&0&0&0&1&0&-1&0\\ 
			0&0&-1&0&0&1&1&-1&0\\
			0&1&0&-1&0&i&0&-i&0
		\end{pmatrix},\ i^2=-1.$
	\end{itemize}
		
	A comparison with our results above shows that the classification presented in \cite{K} is \textbf{not complete}.
	
	The following result describes the automorphisms and derivations our representatives of three-dimensional non-Lie algebras. 

\begin{theorem} The derivations and automorphisms of non-Lie algebras, presented in Theorem, are:\\
	$
	\mathrm{Der}(A_1) = \{ 
	\begin{pmatrix}
		0 & 0 & 0\\
		0 & a & 0\\
		0 & 0 & a
	\end{pmatrix}:\ a\in \mathbb{F}\}, \ 
	\mathrm{Aut}(A_1) = 
	\{ \begin{pmatrix}
		1 & 0 & 0 \\
		0 & t & 0\\
		0 & 0 & t
	\end{pmatrix}:\ 0\neq t\in \mathbb{F}\},$\\
	$Der(A_2(\lambda))=\{a
	\begin{pmatrix}
		-1 & \dfrac{1}{\lambda} & \dfrac{1}{\lambda} \\
		\dfrac{\lambda-1}{\lambda} & 0 & \dfrac{\lambda-1}{\lambda} \\
		-\dfrac{1}{\lambda} & \dfrac{1}{\lambda} & 1
	\end{pmatrix}: a\in \mathbb{F}\},$
	if  $\lambda\neq 0$;
	\ $Der(A_2(0))=\{ a\begin{pmatrix} 0 & 1 & 1\\ -1 & 0 & -1\\ -1 & 1 & 0\end{pmatrix}: a\in \mathbb{F}\},$
	$Aut(A_2(\lambda ))=
	\{I_3\}$,\\
	$Der(A_3) =\{ a\begin{pmatrix}
		-1 & 0 & 0\\
		1 & 0 & 0\\
		-1 & 1 & 1
	\end{pmatrix}: a\in\mathbb{F}\}, Aut(A_3)=\{\begin{pmatrix}1/t&0&0\\1-1/t&1&0\\1/t-1&t-1&t\end{pmatrix}:\ t\in\mathbb{F}^*\},$\\
	$ Der(A_4(\lambda)) =\{ a\begin{pmatrix} 0 & 0 & 0\\ 0 & -1 & 0\\ 0 & 0 & 1\end{pmatrix},\ a\in\mathbb{F}\},\  where\  \lambda\neq -1,\\ Aut(A_4(\lambda))=\{\begin{pmatrix}1&0&0\\0&t&0\\0&0&1/t\end{pmatrix}:\ t\in\mathbb{F}^*\}$,\  where\  $\lambda\neq -1$,\\ 
	$Der(A_5) = \{a\begin{pmatrix} 0 & 0 & 0\\ 0 & 0 & 1\\ 0 & 0 & 0\end{pmatrix}:\ a\in\mathbb{F}\},\\ Aut(A_5)=\{\begin{pmatrix}1&0&0\\0&1&t\\0&0&1\end{pmatrix}: t\in\mathbb{F}\}\cup
	\{\begin{pmatrix}1&0&-1\\1&-1&t\\0&0&-1\end{pmatrix}: t\in\mathbb{F}\}.$\\
	$ Der(A_7(\lambda)) =\{\begin{pmatrix}0&0&0\\b(1-\lambda)&a&b\\0&0&0\end{pmatrix}:\ a,b\in\mathbb{F}\}$,\\
	$Aut(A_7(\lambda))=\{I_3\}$, if $\lambda\neq 0$,\
	$Aut(A_7(0))=\{\begin{pmatrix}1 & 0 & 0\\0 & t & 0\\0 & 0 & 1\end{pmatrix}:\ t\in\mathbb{F}^*\}$.
\end{theorem}
The proof consists of solving the equation $XA=A(X\otimes I+I\otimes X)\ (\mbox{respectively}, XA=A(X\otimes X)$) for a third-order matrix $X$, which is required to be invertible in the case of automorphisms, where $A$ is a canonical representative listed in Theorem 2.6.

 \end{document}